\documentclass[11pt]{article} %11pt font
\usepackage[margin=1in]{geometry}% <--- 1 in margin
\usepackage{setspace}
\usepackage[T1]{fontenc}
\usepackage{a4wide}
\usepackage{lipsum}
\usepackage{tikz}

\definecolor{voros}{cmyk}{0,1,1,0.25}

\title{Posets are easily testable}

\author{
Panna T\'imea Fekete\thanks{Institute of Mathematics, E\"otv\"os Lor\'and University, POB 120, H-1518 Budapest, Hungary \\
and HUN-REN Alfr\'ed R\'enyi Institute of Mathematics, POB 127, H-1364 Budapest, Hungary.} \thanks{
E-mail: {\tt fekete.panna.timea@renyi.hu}.}
\and
G\'abor Kun \protect\footnotemark[1]
}

\makeatletter

\newcommand{\shorttitle}{\@title}

\makeatother

\usepackage{amsthm}
\usepackage{amsmath} % provides \substack
\usepackage{dsfont} %provides e.g. \mathds
\usepackage{amssymb, setspace}
\usepackage[all]{xy}
\usepackage{amsrefs}
\usepackage{graphicx}        % provides \resizebox
\usepackage{xcolor} %provides colors
\usepackage{comment}
\usepackage{hyperref}
\usepackage{algorithm}
\usepackage{algpseudocode}
\usepackage{enumitem}
\usepackage{float}

\newtheorem{theorem}{Theorem}[section]
\newtheorem{lemma}[theorem]{Lemma}
\newtheorem{proposition}[theorem]{Proposition}
\newtheorem{corollary}[theorem]{Corollary}

\theoremstyle{definition}

\theoremstyle{remark}
\newtheorem{remark}[theorem]{Remark}
\theoremstyle{remark}
\newtheorem{claim}[theorem]{Claim}
\theoremstyle{remark}

\newtheorem*{claim*}{Claim}

\newenvironment{thmenum}
 {\begin{enumerate}[label=\upshape(\arabic*)]}
 {\end{enumerate}}

\begin{document}

\thispagestyle{empty}
\maketitle

\begin{abstract}\large{Alon and Shapira proved that every monotone class (closed under taking subgraphs) of undirected graphs 
is strongly testable, that is, under the promise that a given graph is either in the class or $\varepsilon$-far from it, there is a test using a constant number of samples (depending on $\varepsilon$ only) that rejects every graph not in the class with probability at least one half, and always accepts a graph in the class. However, their bound on the number of samples is quite large since they heavily rely on Szemer\'edi's regularity lemma. We study the case of posets and show that every monotone class of posets is easily testable, that is, a polynomial (of $\varepsilon^{-1}$) number of samples is sufficient. We achieve this via proving a polynomial removal lemma for posets.

We give a simple classification: for every monotone class of posets, there is an $h$ such that the class is indistinguishable (every large enough poset in one class is $\varepsilon$-close to a poset in the other class) from the class of $C_h$-free posets, where $C_h$ denotes the chain with $h$ elements.
This allows us to test every monotone class of posets using $O(\varepsilon^{-1})$ samples.
The test has a two-sided error, but it is almost complete: the probability of refuting a poset in the class is polynomially small in the size of the poset.

The analogous results hold for comparability graphs, too.}\end{abstract}

\textbf{Keywords:} Property testing, polynomial removal lemma, poset

\section{Introduction}

The relationship between local and global properties of structures is a central theme in combinatorics and computer science. Since the work of Rubinstein and Sudan \cite{RS96}, testing properties by sampling a small number of elements is an emerging research area.
%, see the book of Bhattacharyya and Yoshida \cite{BY}.
A classical result of this kind is the triangle removal lemma by Ruzsa and Szemer\'edi \cite{RS76}, usually stated in the form that if a graph $G$ admits at most $\delta |V(G)|^3$ triangles then it can be made triangle-free by the removal of at most $\varepsilon |V(G)|^2$ edges, where $\delta$ depends only on $\varepsilon$. This can be applied to obtain a combinatorial proof of Roth's theorem \cite{R53} on $3$-term arithmetic progressions, while the hypergraph removal lemma has been used to prove Szemer\'edi's theorem. Removal lemmas were proved for abelian groups by Green \cite{G05}, for linear systems of equations by Kr\'al, Serra and Vena \cite{KSV12}, for local affine-invariant properties by Bhattacharyya, Fischer, Hatami, Hatami and Lovett \cite{BFHHS} and for permutations by Klimo\v{s}ov\'a and Kr\'al \cite{KK14}, and by Fox and Wei \cite{FW18}, as well.

A {\it property} of digraphs is a set of finite digraphs closed under isomorphism. A digraph $G$ is {\it $\varepsilon$-far} from having a property $\Phi$ if any digraph $G'$ on the vertex set $V(G)$ that differs by at most $\varepsilon |V(G)|^2$ edges from $G$ does not have the property $\Phi$ either.
A property $\Phi$ is {\it strongly testable} if for every $\varepsilon>0$ there exists an $f(\varepsilon)$ such that if the digraph $G$ is $\varepsilon$-far from having the property $\Phi$ then the induced directed subgraph on $f(\varepsilon)$ vertices chosen uniformly at random does not have the property $\Phi$ with probability at least one half, and it always has the property if $G$ does.
Alon and Shapira \cite{AS05} proved that every {\it monotone property} of undirected graphs (that is, closed under the removal of edges and vertices) is strongly testable, see Lov\'asz and Szegedy for an analytic approach \cite{LS08}, while R\"odl and Schacht generalized this to hypergraphs \cite{RS09}, see also Austin and Tao \cite{AT10}. Similar results have been obtained for hereditary classes of graphs and other structures, e.g., tournaments and matrices, see Gishboliner for the most recent summary \cite{G20}. We focus on monotone properties and omit the overview of other research directions.

Unfortunately, the dependence on $\varepsilon$ can be quite bad already in the case of undirected graphs: the known upper bounds in the Alon-Shapira theorem are wowzer functions due to the iterated involvement of Szemer\'edi's regularity lemma. Following Alon and Fox \cite{AF15}, we call a property {\it easily testable} if $f(\varepsilon)$ can be bounded by a polynomial of $\frac{1}{\varepsilon}$, else the property is {\it hard}. They showed that both testing perfect graphs and testing comparability graphs are hard. Easily testable properties are quite rare, even triangle-free graphs are hard: Behrend's construction \cite{B46} of sets of integers without $3$-term arithmetic progression leads to a lower bound of magnitude $\varepsilon^{c \log\left(\frac{1}{\varepsilon}\right)}$. Alon proved that $H$-freeness is easily testable in the case of undirected graphs if and only if $H$ is bipartite. For forbidden induced subgraphs, Alon and Shapira gave a characterization \cite{AS06}, where there are very few easy cases. Testability is usually hard for hypergraphs studied by Gishboliner and Shapira \cite{GS22} and ordered graphs investigated by Gishboliner and Tomon \cite{GT21}. An interesting class of properties that are easy to test are semialgebraic hypergraphs, see Fox, Pach and Suk \cite{FPZ16}.
Surprisingly, $3$-colorability and, in general, "partition problems" turned out to be easily testable, see Goldreich, Goldwasser and Ron \cite{GGR98}. 
Even a conjecture to draw the borderline between easy and hard properties seems beyond reach.

The goal of this paper is to study testability of finite posets as special digraphs. By a poset, we mean a set equipped with a partial order $\prec$ that is anti-reflexive and transitive. Alon, Ben-Eliezer and Fischer \cite{ABF17} proved that hereditary (closed under induced subgraphs) classes of ordered graphs are strongly testable. This implies the removal lemma for posets and that monotone classes of posets are strongly testable in the following way. We consider a linear extension $<$ of the ordering $\prec$ of the poset $P$. To every poset with a linear ordering, we can associate the graph on its base set, where distinct elements $x<y$ are adjacent if $x \prec y$ in the poset. A graph with a linear ordering is associated with a poset if and only if it has no induced subgraph with two edges on three vertices, where the smallest and largest vertices are not adjacent.
An alternative to the application of this general result is to follow the proof of Alon and Shapira \cite{AS05} using the poset version of Szemer\'edi's regularity lemma proved by Hladk\'y, M\'ath\'e, Patel and Pikhurko \cite{HMPP15}.

We show that monotone classes of posets (closed under taking subposets) are easily testable. This is equivalent to the following removal lemma with polynomial bounds.

Throughout this paper, we work with finite posets. The {\it height} of a poset $P$ is the length of its longest chain, while the {\it width} is the size of the largest antichain, denoted by $h(P)$ and $w(P)$, respectively.
The chain with $h$ elements is denoted by $C_h$.
Given two posets $P, Q$, a mapping $f: Q \rightarrow P$ is a homomorphism if it is order-preserving, i.e., $f(x) \prec f(y)$ for every $x \prec y$. The probability that a uniform random mapping from $Q$ to $P$ is a homomorphism is denoted by $t(Q,P)$, which we often refer to as the homomorphism density. A poset $P$ is called {\it $Q$-free} if it does not contain $Q$ as a (not necessarily induced) subposet.

\begin{theorem}\label{thm:posetremoval} [Polynomial removal lemma for posets]
Consider an $\varepsilon>0$ and a finite poset $Q$ of height at least two. For every finite poset $P$, if $t(Q,P) <
\left(\frac{\varepsilon}{2}\right)^{h(Q)w(Q)^2}$ then there exists a $Q$-free (moreover, $C_{h(Q)}$-free) subposet of $P$ obtained by the removal of at most $\varepsilon |P|^2$ edges.
\end{theorem}

We show that $Q$-free posets are easily testable.

\begin{algorithm}[H]
\caption{Basic test for $Q$-free posets}\label{test1}
\begin{algorithmic}
\Require{the poset $P$}
\State $P' \gets $ induced subposet on $|Q|$ elements chosen uniformly at random
\If { $Q$ is a subposet of $P'$} Reject $P$
\Else{} Accept $P$
\EndIf
\end{algorithmic}
\end{algorithm}

This test always accepts a $Q$-free poset, and rejects a poset $P$ with probability at least $t(Q,P)$. If $t(Q,P) < \left(\frac{\varepsilon}{2}\right)^{h(Q)w(Q)^2}$, then by Theorem~\ref{thm:posetremoval} $P$ is not $\varepsilon$-far from being $Q$-free. If $t(Q,P) \geq \left(\frac{\varepsilon}{2}\right)^{h(Q)w(Q)^2}$, then it is sufficient to iterate this test 
$\frac{1}{t(Q,P)} \leq \left(\frac{2}{\varepsilon}\right)^{h(Q)w(Q)^2}$ times independently (i.e., taking $f(\varepsilon) = \left(\frac{2}{\varepsilon}\right)^{h(Q)w(Q)^2} |Q|$ in the definition of easy testability) to reject a poset $\varepsilon$-far from being $Q$-free with probability at least $1-\left(1-t(Q,P)\right)^{\frac{1}{t(Q,P)}} > \frac{1}{2}$. The inequality holds since $0 < t(Q,P) \leq 1$, the function $t \mapsto1-\left(1-t\right)^{\frac{1}{t}}$ is monotone increasing on $(0,1]$ and $\lim_{t\rightarrow 0} 1-\left(1-t\right)^{\frac{1}{t}} = 1-\frac{1}{e}$.

We will consider the family of (possibly infinitely many) finite posets not in the class. To state our precise result, we define the height and width of a set of posets $\mathcal{P}$ as
\[h(\mathcal{P}) = \min_{P\in \mathcal{P}} h(P) \quad \quad \quad w(\mathcal{P}) = \min_{\substack{P\in \mathcal{P}:\\ h(P)=h(\mathcal{P})}} w(P).\]

\begin{corollary} \label{cor:posettesting} [Easy testability for monotone classes of posets]
Consider a family of finite posets $\mathcal{P}$ with
$h(\mathcal{P})\geq 2$. Let $Q\in \mathcal{P}$ with height $h(Q) = h(\mathcal{P})$ and width $w(Q) = w(\mathcal{P})$.
For every $\varepsilon>0$ and finite poset $P$,
if
$t(Q,P) < \left(\frac{\varepsilon}{2}\right)^{h(\mathcal{P})w(\mathcal{P})^2}$ then there exists a $\mathcal{P}$-free (moreover, $C_{h(\mathcal{P})}$-free) subposet of $P$ obtained by the removal of at most $\varepsilon |P|^2$ edges.
\end{corollary}
Observe that by Theorem~\ref{thm:posetremoval} there exists a $C_{h(Q)}$-free subposet of $P$ obtained by the removal of at most $\varepsilon |P|^2$ edges. Since every poset in $\mathcal{P}$ contains $C_{h(\mathcal{P})}$, this subposet is also $\mathcal{P}$-free, hence Corollary~\ref{cor:posettesting} holds.

Chains will play an important role in more efficient tests for monotone classes of posets:
we give a simple classification of these classes from the testing point of view. Two properties $\Phi_1$ and $\Phi_2$ of posets are {\it indistinguishable} if for every $\varepsilon>0$ and $i=1,2$ there exists $N$ such that for every poset $P$ on at least $N$ elements with property $\Phi_i$ there exists a poset $P'$ on the same set with property $\Phi_{3-i}$ obtained by changing at most $\varepsilon |P|^2$ edges of $P$. Since we are interested in monotone properties, we only need to allow deleting edges and not adding them.

\begin{theorem} \label{thm:indistinguish} [Indistinguishability]
Consider a family of finite posets $\mathcal{P}$, set $h=h(\mathcal{P}) \geq 2$ and $w=w(\mathcal{P})$. The class of
$\mathcal{P}$-free posets and the class of $C_h$-free posets are indistinguishable.
Namely, every $C_h$-free poset is $\mathcal{P}$-free, and if a poset $P$ is $\mathcal{P}$-free then it has a $C_h$-free subposet obtained by the removal of 
at most $2 \left(\frac{h^2 w^2}{|P|}\right)^{\frac{1}{hw^2}}|P|^2$ edges. 
\end{theorem}

In other words, for every $\mathcal{P}$-free poset $P$ on at least $N = h^2w^2(\varepsilon/2)^{-hw^2}$ elements there exists a $C_h$-free (not necessarily induced) subposet $P'$ obtained by the removal of at most $\varepsilon |P|^2$ edges.

Theorem \ref{thm:indistinguish} motivates a better understanding of the removal lemma for chains and the testing of $C_h$-free posets.
First, we study the basic test with one-sided error.
We can also use this test for $C_h$-free posets to test $\mathcal{P}$-free posets, where $h=h(\mathcal{P})$. This test is not complete, but the probability of rejecting a $\mathcal{P}$-free poset turns out to be negligible, 
$2 \left(\frac{h^2 w^2}{|P|}\right)^{\frac{1}{hw^2}} \cdot {\binom{h}{2}}$, where $w=w(\mathcal{P})$, since every copy of $C_h$ should contain one of the edges removed in Theorem~\ref{thm:indistinguish}.
If we iterate the test $\left(\frac{2}{\varepsilon}\right)^{h}$ times independently, then the probability of accepting a poset $\varepsilon$-far from being $\mathcal{P}$-free is at most one half by Theorem~\ref{thm:posetremoval}. On the other hand, the probability of rejecting a poset that is $\mathcal{P}$-free is at most $2 \left(\frac{h^2 w^2}{|P|}\right)^{\frac{1}{hw^2}} {\binom{h}{2}}\left(\frac{2}{\varepsilon}\right)^{h}$, and this is negligible if $\varepsilon, h, w$ are fixed and $|P|$ is large enough.

We can get a more efficient test by sampling larger subposets instead of iterating the basic test with a constant number of samples.

\begin{algorithm}[H]
\caption{Subposet test for $C_h$-free posets with $s$ samples}\label{test2}
\begin{algorithmic}
\Require{the poset $P$}
\State $P' \gets $ induced subposet of $s$ elements chosen uniformly at random 
\If {$C_h$ is a subposet of $P'$} Reject $P$
\Else{} Accept $P$
\EndIf
\end{algorithmic}
\end{algorithm}

It turns out that sampling $s=\left\lceil \frac{4\log(h)+4}{2\varepsilon} \right\rceil$ elements is enough to reject posets $\varepsilon$-far from being $C_h$-free with probability at least one half, while we always accept $C_h$-free posets. 

By Theorem~\ref{thm:indistinguish} this test can also be used for testing $\mathcal{P}$-free posets, where $h(\mathcal{P})=h$: it rejects posets $\varepsilon$-far from $\mathcal{P}$-free with probability at least one half
at the price of allowing the error of rejecting a $\mathcal{P}$-free poset with negligible probability.

\begin{theorem} \label{thm:optsample} [The subposet test]
Let $h\geq2$ be an integer,
$\varepsilon>0, c>0$ and $P$ a finite poset. If $P$ is $\varepsilon$-far from being $C_h$-free then a random subset of $\left\lceil \frac{4\log(h)+4c+1}{2\varepsilon} \right\rceil$ elements chosen independently and uniformly at random contains a copy of $C_h$ with probability at least $1-e^{-c}$.
\end{theorem}

Observe that being $\varepsilon$-far from every $C_h$-free poset guarantees that $\varepsilon$ is small, so the number of samples will be large enough.

\begin{remark}
Every poset $P$ is $\frac{1}{2h-2}$-close to be $C_h$-free.
\end{remark}

\begin{proof}
Every poset can be extended to a linear ordering. Partition the poset $P$ into $(h-1)$ intervals of equal size and remove the edges inside the intervals: this gives a $C_h$-free poset $\frac{1}{2h-2}$-close to $P$.
\end{proof}

For any fixed $h$, our bound gives the right order of magnitude (in $\varepsilon$) on the necessary number of samples for one-sided testing of $C_h$-free posets, see Proposition~\ref{prop:subposet_sharp}.

The comparability graph $G$ associated with a poset $P$ has vertex set $V(G)=P$ and edge set $E(G)=\{(x,y): x \prec y \text{ or } y \prec x\}$.
Alon and Fox proved that it is hard to test if a given graph is a comparability graph \cite{AF15}.
However, under the promise that the input graph is a comparability graph, we can test monotone classes, even though we do not know the underlying poset. All of our results apply to testing monotone classes of comparability graphs, see Section~\ref{section:cograph}.

In a subsequent work, we prove that the exact degree is $(h-1)$ in the polynomial removal lemma for chains (and many other structures). Proposition~ \ref{prop:removal_sharp}
shows that this is sharp. The proof is too technical for this paper to detail here.

In Section~\ref{section:testingchains}, we prove the polynomial removal lemma for chains and Theorem~\ref{thm:optsample}. Section~\ref{section:testing} contains the proofs of Theorem~\ref{thm:posetremoval} and Theorem~\ref{thm:indistinguish}.
Section~\ref{section:cograph} discusses our results on comparability graphs.

\section{Testing chains}\label{section:testingchains}

First, we prove a removal lemma for chains. 

\begin{lemma}~\label{lemma:chainremoval}[Removal lemma for chains]
For every $\varepsilon>0$, positive integer $h \geq 2$ and every finite poset $P$, if $t(C_h,P)< \big(\frac{\varepsilon}{2}\big)^h$ then there exists a $C_h$-free subposet of $P$ obtained by the removal of at most $\varepsilon |P|^2$ edges of $P$.
\end{lemma}

Polynomial removal lemmas for directed paths have already been obtained by Alon and Shapira \cite{AS03}, but their bound is $O(\varepsilon^{h^2})$. 
We could use their result to get a removal lemma for chains with a worse polynomial bound.
However, we improve their bound to degree $h$. This is almost the exact degree, as the following example shows.

\begin{proposition}\label{prop:removal_sharp}
    Consider the integer $h \geq 2$ and $\varepsilon > 0$ such that $\varepsilon^{-1}$ is an integer. Let $P$ be a poset that consists of $\varepsilon^{-1}$ chains of equal size at least $h$ and divisible by $(h-1)$.
    \begin{thmenum}
        \item \label{prop:removal_sharp_1}
        Every subposet obtained by the removal of less than $\frac{1}{2}(\frac{\varepsilon}{h-1}|P|^2-|P|)$ edges from $P$ contains $C_h$ as a subposet, hence $P$ is at least $(\frac{\varepsilon}{2h-2}-\frac{1}{2|P|})$-far from being $C_h$-free.
        \item \label{prop:removal_sharp_2} The inequality $t(C_h,P) < \frac{\varepsilon^{h-1}}{h!}$ holds. 
    \end{thmenum}
\end{proposition}
\begin{proof}
\ref{prop:removal_sharp_1} The comparability graph of $P$ is the union of $\varepsilon^{-1}$ complete graphs $K_{\varepsilon |P|}$. If $P'$ is a $C_h$-free subposet of $P$, then the corresponding comparability graph is $K_{h}$-free. By Turán's theorem we have to remove at least $(h-1) \binom{\frac{\varepsilon|P|}{h-1}}{2}$ edges from $K_{\varepsilon|P|}$ in order to obtain a $K_h$-free graph. Now \ref{prop:removal_sharp_1} follows, since $\varepsilon^{-1}(h-1) \binom{\frac{\varepsilon|P|}{h-1} }{2}=\frac{\varepsilon}{2h-2}|P|^2-\frac{1}{2}|P|$.

\ref{prop:removal_sharp_2} The probability that all of the $h$ elements are mapped to the same chain is $\varepsilon^{h-1}$. Note that any homomorphism $C_h \rightarrow P$ maps $C_h$ onto an $h$ element chain in $P$, since $P$ is anti-reflexive. The conditional probability that such a bijection preserves the order of the elements is $\frac{1}{h!}$.% Hence the inequality $t(C_h,P) = \varepsilon^{-1} \binom{\varepsilon |P| + h-1}{h} |P|^{-h}< \varepsilon^{-1} \frac{\left(\varepsilon |P| + h\right)^h}{h!} |P|^{-h} = \frac{\varepsilon^{h-1}}{h!} \left(1+\frac{h}{\varepsilon |P|}\right)^h$ holds.
\end{proof}

We consider a linear extension $<$ of the ordering $\prec$ of the poset $P$. We may assume that the set of elements of $P$ is $[|P|]=\{1,2\dots, |P|\}$, and the linear ordering $<$ is the ordering of the integers.

The algorithm defines a rank function $r$ on the set of elements, such that if $r(y)=k+1$ for some element $y$, then it has 'many' predecessors $x\prec y$ with $r(x)=k$. Hence, it has 'many' chains $C_{k+1}$ ending at $y$.

\begin{algorithm}[H]
\caption{Rank function $r$}\label{alg:rank}
\begin{algorithmic}
\Require{$\gamma>0$, poset $P$ on $[|P|]$, where if $x\prec y$ then $x<y$}
\For {$y =1, \dots, |P|$}
\If{ 
$ \quad \exists  k: \big|\{x: x\prec y, r(x)=k \}\big| \geq \gamma |P|$} 
 \State $r(y) \gets 1 + \max\left\{ k: \big|\{x: x\prec y, r(x)=k \}\big| \geq \gamma |P|\right\}$
\Else
\State $r(y) \gets 1$ 
\EndIf
\EndFor
\Ensure{Rank function $r: P\rightarrow \mathbb{Z}_+$}
\end{algorithmic}
\end{algorithm}

The following algorithm will remove the edges to get a $C_h$-free poset.

\begin{algorithm}[H]
\caption{Edge removal using the rank function $r$}\label{alg:edgeremoval}
\begin{algorithmic}
\Require{$\gamma>0$, $h \in \mathds{Z_+}$, poset $P$ on $[|P|]$, where if $x\prec y$ then $x<y$}
\State \Call{Algorithm~\ref{alg:rank}}{$\gamma$, $P$}
\For{$x \prec y$}
\If{
$r(x) = r(y)$}
\State $E(P) \gets E(P) \setminus\{(x,y)\}$
\ElsIf {$r(y) \geq h$}
\State $E(P) \gets E(P) \setminus\{(x,y)\}$
\EndIf
\EndFor
\State $P' \gets P$
\Ensure{$P'$ on vertex set $[|P|]$, edge set $E(P') \subseteq E(P)$}
\end{algorithmic}
\end{algorithm}

\begin{figure}
    \centering
\begin{tikzpicture}[line cap=round,line join=round,
x=0.75cm,y=0.75cm, scale=0.35]

\draw [line width=1pt] (-7,-4) -- (-7,-1);
\draw [line width=1pt] (-7,-4) -- (-3,-1);

\draw [line width=1pt] (-3,-4) -- (-7,-1);
\draw [line width=1pt] (-3,-4) -- (-3,-1);
\draw [line width=1pt] (-3,-4) -- (1,-1);

\draw [line width=1pt] (1,-4) -- (-7,-1);
\draw [line width=1pt] (1,-4) -- (-3,-1);
\draw [line width=1pt] (1,-4) -- (5,-1);

\draw [line width=1pt] (5,-4) -- (-3,-1);
\draw [line width=1pt] (5,-4) -- (9,-1);

\draw [line width=1pt] (9,-4) -- (1,-1);

\draw [line width=1pt] (-7,-1) -- (-5,2);
\draw [line width=1pt] (-7,-1) -- (3,2);

\draw [line width=1pt] (-3,-1) -- (-1,2);

\draw [line width=1pt] (1,-1) -- (-5,2);
\draw [line width=1pt] (1,-1) -- (3,2);

\draw [line width=1pt] (5,-1) -- (-1,2);
\draw [line width=1pt] (5,-1) -- (7,2);
\draw [line width=1pt] (5,-1) -- (3,2);

\draw [line width=1pt] (9,-1) -- (7,2);

\draw [line width=1pt] (-5,2) -- (-7,5);
\draw [line width=1pt] (-5,2) -- (-3,5);
\draw [line width=1pt] (-5,2) -- (1,5);

\draw [line width=1pt] (-1,2) -- (-7,5);
\draw [line width=1pt] (-1,2) -- (-3,5);
\draw [line width=1pt] (-1,2) -- (1,5);
\draw [line width=1pt] (-1,2) -- (5,5);

\draw [line width=1pt] (3,2) -- (-3,5);
\draw [line width=1pt] (3,2) -- (1,5);
\draw [line width=1pt] (3,2) -- (5,5);

\draw [line width=1pt] (7,2) -- (9,5);
\draw [line width=1pt] (7,2) -- (5,5);

\draw [line width=1pt] (-7,5) -- (-3,8);

\draw [line width=1pt] (-3,5) -- (-3,8);
\draw [line width=1pt] (-3,5) -- (1,8);

\draw [line width=1pt] (1,5) -- (-3,8);
\draw [line width=1pt] (1,5) -- (5,8);
\draw [line width=1pt] (1,5) -- (1,8);

\draw [line width=1pt] (5,5) -- (1,8);
\draw [line width=1pt] (5,5) -- (5,8);

\draw [line width=1pt] (9,5) -- (5,8);

\draw (-7,-4) node[anchor=north] {1};
\draw (-3,-4) node[anchor=north] {1};
\draw (1,-4) node[anchor=north] {1};
\draw (5,-4) node[anchor=north] {1};
\draw (9,-4) node[anchor=north] {1};

\draw (-7,-1) node[anchor=south] {2};
\draw (-3,-1) node[anchor=south] {2};
\draw (1,-1) node[anchor=south] {2};
\draw (5,-1) node[anchor=south] {1};
\draw (9,-1) node[anchor=south] {1};

\draw (-5,2) node[anchor=south] {3};
\draw (-1,2) node[anchor=south] {2};
\draw (3,2) node[anchor=south] {3};
\draw (7,2) node[anchor=south] {2};

\draw (-7,5) node[anchor=south east] {3};
\draw (-3,5) node[anchor=south east] {4};
\draw (1,5) node[anchor=south west] {4};
\draw (5,5) node[anchor=south west] {3};
\draw (9,5) node[anchor=south] {2};

\draw (-3,8) node[anchor=south] {5};
\draw (1,8) node[anchor=south] {5};
\draw (5,8) node[anchor=south] {4};

\begin{scriptsize}
\draw [fill=black] (-7,-4) circle (5pt);
\draw [fill=black] (-3,-4) circle (5pt);
\draw [fill=black] (1,-4) circle (5pt);
\draw [fill=black] (5,-4) circle (5pt);
\draw [fill=black] (9,-4) circle (5pt);
\draw [fill=black] (-7,-1) circle (5pt);
\draw [fill=black] (-3,-1) circle (5pt);
\draw [fill=black] (1,-1) circle (5pt);
\draw [fill=black] (5,-1) circle (5pt);
\draw [fill=black] (9,-1) circle (5pt);
\draw [fill=black] (-5,2) circle (5pt);
\draw [fill=black] (-1,2) circle (5pt);
\draw [fill=black] (3,2) circle (5pt);
\draw [fill=black] (7,2) circle (5pt);
\draw [fill=black] (-3,8) circle (5pt);
\draw [fill=black] (-7,5) circle (5pt);
\draw [fill=black] (-3,5) circle (5pt);
\draw [fill=black] (5,5) circle (5pt);
\draw [fill=black] (9,5) circle (5pt);
\draw [fill=black] (1,8) circle (5pt);
\draw [fill=black] (5,8) circle (5pt);
\draw [fill=black] (1,5) circle (5pt);
\end{scriptsize}
\end{tikzpicture}
\quad
%\begin{tikzpicture}[line cap=round,line join=round,%>=triangle 45,
%x=0.75cm,y=0.75cm, scale=0.35]
%    \draw [-to,line width=1pt] (0,-6) -- (0,8);
%    \draw [-to,line width=1pt] (1,-6) -- (1,8);
%\end{tikzpicture}
\quad
\begin{tikzpicture}[line cap=round,line join=round,%>=triangle 45,
x=0.75cm,y=0.75cm, scale=0.35]

\draw [line width=1pt] (-7,-4) -- (-7,-1);
\draw [line width=1pt] (-7,-4) -- (-3,-1);
\draw [line width=1pt] (-7,-4) -- (-1,2);

\draw [line width=1pt] (-3,-4) -- (-7,-1);
\draw [line width=1pt] (-3,-4) -- (-3,-1);
\draw [line width=1pt] (-3,-4) -- (1,-1);
\draw [line width=1pt] (-3,-4) -- (-1,2);

\draw [line width=1pt] (1,-4) -- (-7,-1);
\draw [line width=1pt] (1,-4) -- (-3,-1);
%\draw [line width=1pt] (1,-4) -- (5,-1);
\draw [line width=1pt] (1,-4) -- (-1,2);
\draw [line width=1pt] (1,-4) -- (3,2);
\draw [line width=1pt] (1,-4) -- (7,2);

\draw [line width=1pt] (5,-4) -- (-3,-1);
%\draw [line width=1pt] (5,-4) -- (9,-1);
\draw [line width=1pt] (5,-4) -- (7,2);
\draw [line width=1pt] (5,-4) -- (-1,2);

\draw [line width=1pt] (9,-4) -- (1,-1);

\draw [line width=1pt] (-7,-1) -- (-5,2);
\draw [line width=1pt] (-7,-1) -- (3,2);
\draw [line width=1pt] (-7,-1) -- (-7,5);
\draw [line width=1pt] (-7,-1) to[out=10,in=-90]  (5,5);

%\draw [line width=1pt] (-3,-1) -- (-1,2);
\draw [line width=1pt] (-3,-1) to[out=150,in=-90] (-7,5);
\draw [line width=1pt] (-3,-1) -- (-3,5);
\draw [line width=1pt] (-3,-1) to[out=30,in=-90] (1,5);

\draw [line width=1pt] (1,-1) -- (-5,2);
\draw [line width=1pt] (1,-1) -- (3,2);
\draw [line width=1pt] (1,-1) -- (-7,5);
\draw [line width=1pt] (1,-1) to[out=30,in=-90]  (5,5);

\draw [line width=1pt] (5,-1) -- (-1,2);
\draw [line width=1pt] (5,-1) -- (7,2);
\draw [line width=1pt] (5,-1) -- (3,2);
\draw [line width=1pt] (5,-1) -- (5,5);

\draw [line width=1pt] (9,-1) -- (7,2);

%\draw [line width=1pt] (-5,2) -- (-7,5);
\draw [line width=1pt] (-5,2) -- (-3,5);
\draw [line width=1pt] (-5,2) -- (1,5);
%\draw [line width=1pt] (-5,2) -- (-3,8);
\draw [line width=1pt] (-5,2) to[out=10,in=-115]  (5,8);

\draw [line width=1pt] (-1,2) -- (-7,5);
\draw [line width=1pt] (-1,2) -- (-3,5);
\draw [line width=1pt] (-1,2) -- (1,5);
\draw [line width=1pt] (-1,2) -- (5,5);
\draw [line width=1pt] (-1,2) -- (5,8);

\draw [line width=1pt] (3,2) -- (-3,5);
\draw [line width=1pt] (3,2) -- (1,5);
%\draw [line width=1pt] (3,2) -- (5,5);
%\draw [line width=1pt] (3,2) -- (1,8);
\draw [line width=1pt] (3,2) -- (5,8);

%\draw [line width=1pt] (7,2) -- (9,5);
\draw [line width=1pt] (7,2) -- (5,5);

%\draw [line width=1pt] (-7,5) -- (-3,8);

%\draw [line width=1pt] (-3,5) -- (-3,8);
%\draw [line width=1pt] (-3,5) -- (1,8);

%\draw [line width=1pt] (1,5) -- (-3,8);
%\draw [line width=1pt] (1,5) -- (5,8);
%\draw [line width=1pt] (1,5) -- (1,8);

%\draw [line width=1pt] (5,5) -- (1,8);
\draw [line width=1pt] (5,5) -- (5,8);

\draw [line width=1pt] (9,5) -- (5,8);

\draw (-7,-4) node[anchor=north] {1};
\draw (-3,-4) node[anchor=north] {1};
\draw (1,-4) node[anchor=north] {1};
\draw (5,-4) node[anchor=north] {1};
\draw (9,-4) node[anchor=north] {1};

\draw (-7,-1) node[anchor=south] {2};
\draw (-3,-1) node[anchor=south] {2};
\draw (1,-1) node[anchor=south] {2};
\draw (5,-1) node[anchor=south] {1};
\draw (9,-1) node[anchor=south] {1};

\draw (-5,2) node[anchor=south] {3};
\draw (-1,2) node[anchor=south] {2};
\draw (3,2) node[anchor=south] {3};
\draw (7,2) node[anchor=south] {2};

\draw (-7,5) node[anchor=south east] {3};
\draw (-3,5) node[anchor=south east] {4};
\draw (1,5) node[anchor=south west] {4};
\draw (5,5) node[anchor=south west] {3};
\draw (9,5) node[anchor=south] {2};

\draw (-3,8) node[anchor=south] {5};
\draw (1,8) node[anchor=south] {5};
\draw (5,8) node[anchor=south] {4};

\begin{scriptsize}
\draw [fill=black] (-7,-4) circle (5pt);
\draw [fill=black] (-3,-4) circle (5pt);
\draw [fill=black] (1,-4) circle (5pt);
\draw [fill=black] (5,-4) circle (5pt);
\draw [fill=black] (9,-4) circle (5pt);
\draw [fill=black] (-7,-1) circle (5pt);
\draw [fill=black] (-3,-1) circle (5pt);
\draw [fill=black] (1,-1) circle (5pt);
\draw [fill=black] (5,-1) circle (5pt);
\draw [fill=black] (9,-1) circle (5pt);
\draw [fill=black] (-5,2) circle (5pt);
\draw [fill=black] (-1,2) circle (5pt);
\draw [fill=black] (3,2) circle (5pt);
\draw [fill=black] (7,2) circle (5pt);
\draw [fill=black] (-3,8) circle (5pt);
\draw [fill=black] (-7,5) circle (5pt);
\draw [fill=black] (-3,5) circle (5pt);
\draw [fill=black] (5,5) circle (5pt);
\draw [fill=black] (9,5) circle (5pt);
\draw [fill=black] (1,8) circle (5pt);
\draw [fill=black] (5,8) circle (5pt);
\draw [fill=black] (1,5) circle (5pt);
\end{scriptsize}
\end{tikzpicture}
\caption{Example for Algorithm~\ref{alg:rank} on the left and Algorithm~\ref{alg:edgeremoval} on the right with $h=5, \gamma = \frac{1}{11}$. The Hasse diagram of the poset $P$ is on the left, the ranks are written on the elements, and the Hasse diagram of $P'$ is on the right.}
  \label{fig:alg1_P}
\end{figure}
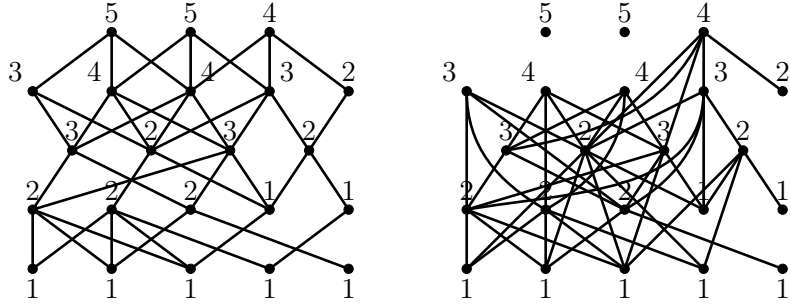

\noindent
{\bf Analysis of Algorithm~\ref{alg:edgeremoval}:}

\begin{claim}\label{claim:alg1} The following holds.
\begin{thmenum}
\item \label{claim_poset} {The output $P'$ is a poset.}

\item \label{claim_chainfree}
{The output poset $P'$ is $C_h$-free.}

\item \label{claim_removed}
{The number of edges $x\prec y$ removed such that $r(x) = r(y)$ is at most $\gamma |P|^2$.}

\item \label{claim_removed_h}
{If the number of elements with rank $r(y) \geq h$ is at most $\gamma|P|$, then the number of edges removed by Algorithm~\ref{alg:edgeremoval}} in order to get a $C_h$-free poset is at most $2\gamma|P|^2$.
\end{thmenum}
\end{claim}

\begin{proof}
\ref{claim_poset} If $x,y,z$ are distinct elements in $P$ with $(x,y) \in E(P')$ and $(y,z) \in E(P')$, then $(x,z) \in E(P)$, and $r(z)<h$, $r(x)<r(y)<r(z)$. Hence $(x,z) \in E(P')$.

\ref{claim_chainfree} Let $x,y$ be distinct elements in $P$ with $(x,y) \in E(P')$. Note that $r(x)\leq r(y)$ by the transitivity in posets, hence $r$ is non-decreasing on every chain in $P$. Every edge with $r(x)=r(y)$ has been removed. Thus, $r$ is strictly increasing on every chain in $P'$. The poset $P'$ is $C_h$-free since the edges ending at those elements, where $r$ is at least $h$, have been removed. 

\ref{claim_removed}
For every $y$, the number of $x \prec y$ with $r(x) = r(y)$ can be at most $\gamma|P|$, else $r(y)$ would be greater than $r(x)$.
So, the number of such removed edges is at most $\gamma|P|^2$.

\ref{claim_removed_h}
This is a straightforward consequence of the algorithm and \ref{claim_removed}.
\end{proof}

\begin{proof}(of Lemma~\ref{lemma:chainremoval})
We run Algorithm~\ref{alg:edgeremoval} with $h, P$ and $\gamma=\frac{\varepsilon}{2}$.
\begin{claim*}
    If $t(C_h,P) < \gamma^h$, then the number of elements with rank $r(y)\geq h$ is strictly less than $\gamma |P|$. In particular, there is no element with rank $(h+1)$.
\end{claim*}
\begin{proof}
\renewcommand\qedsymbol{$\blacksquare$}
    Observe that there are at least $(\gamma |P|)^{r(x)-1}$ chains on $r(x)$ elements ending at $x$ for every $x$ such that $r$ is strictly increasing on these chains.

There is no element where $r$ takes value $(h+1)$ since such an element would be the end of at least $\gamma |P|^h$ chains on at least $(h+1)$ elements, but we do not have so many different chains of length $h$.
By the same reason, the number of elements, where $r$ takes value $h$, is strictly less than $\gamma|P|$.
\end{proof}

\ref{claim_removed_h} of Claim~\ref{claim:alg1} proves the lemma.
\end{proof}

Now we use the rank function defined by Algorithm~\ref{alg:rank} to optimize the number of samples to test $C_h$-free posets.

\begin{proof}(of Theorem~\ref{thm:optsample})
We consider again a linear extension $<$ of the ordering $\prec$ of the poset $P$. We might assume that the set of elements of $P$ is $[|P|]=\{1,2\dots, |P|\}$, and the linear ordering $<$ is the ordering of the integers.
We define $r:P \mapsto \mathbb{Z}_+$ using Algorithm~\ref{alg:rank} with $\gamma = \frac{\varepsilon}{2}$.

Let $X$
be a subset of $\left\lceil\frac{4\log(h) + 4c+1}{2\varepsilon} \right\rceil = \left\lceil\frac{\log(h) + c}{\gamma} + \frac{1}{4\gamma} \right\rceil$ elements chosen uniformly at random from $P$. We prove that with probability at least $(1-e^{-c})$ there is a chain with elements $x_{h} \prec \dots \prec x_{2} \prec x_{1}$ such that $r(x_{k}) = h-k+1$ for all $k\in [h]$. We will find these elements one by one, starting with $x_{1}$.

We show that there are at least $\gamma |P|$ elements $x \in P$ with $r(x)=h$. Suppose for a contradiction that there are less.
Then running Algorithm~\ref{alg:edgeremoval} gives a $C_h$-free poset and by \ref{claim_removed_h} of Claim~\ref{claim:alg1} we removed at most $\varepsilon |P|^2$ edges, contradicting that $P$ was $\varepsilon$-far from being $C_h$-free.

Thus, the probability that we do not choose any element with $r(x)=h$ into the set $X$ is at most $(1-\gamma)^{\gamma^{-1} (\log(h)+c +1/4)}<\frac{e^{-c}}{h}$. Denote by $x_1$ the smallest element (in the linear extension) such that $r(x_{1}) = h$, if there is such an element.

\begin{claim*}
Consider ${x_1}, {x_2}, \dots, {x_k}$ for $k<h$ such that for every $\ell \in \{2, 3, \dots, k\}$ 
 the element $x_{\ell}$ is the smallest (in the linear extension) such that $ r(x_{\ell})=h-\ell+1$ and $x_{\ell} \prec x_{\ell-1}$. Then the conditional distribution on the choice of $x_1, \dots, x_k$ of the other elements of $X$ 
 is uniform on the set
$$S_k := \left\{ x \in P \setminus \{x_{1}, x_2, \dots, x_{k}\}: \forall \ell \in [k] \text{ if } x < x_{\ell} \text{ then } \{r(x) \neq h-\ell+1\} \vee \{x \nprec x_{\ell-1} \}\right\}.$$
\end{claim*}

\begin{proof}
\renewcommand\qedsymbol{$\blacksquare$}
Note that $X \subseteq S_k \cup \{x_{1}, x_{2}, \dots, x_{k}\}$: else for the smallest $x$ (in the linear extension) such that $x \notin S_k \cup \{x_{1}, x_{2}, \dots, x_{k}\}$ there would be an $\ell$ such that $x<x_{\ell}$, $r(x)=h-\ell+1$ and $x \prec x_{{\ell-1}}$. Hence, we should have chosen $x$ instead of $x_{\ell}$.

On the other hand, the set $X$ could be $S'\cup \{x_{1}, x_{2}, \dots, x_{k}\}$ for any subset $S' \subseteq S_k$ of size $\left\lceil\frac{\log(h) + c}{\gamma} + \frac{1}{4\gamma} \right\rceil-k$.
Since the conditional distribution of $X$ is uniform on these sets, the claim follows.
\end{proof}

Now we show that a suitable $x_{k+1}$ exists with probability at least $1-\frac{e^{-c}}{h}$.

There are at least $\gamma |P|$ elements $x\in P$ (in particular, $x\in P\setminus \{x_{1}, x_2, \dots, x_{k}\}$ by the partial ordering) such that $x \prec x_{k}$ and $r(x)=h-k$ by the definition of the rank function. Let us denote these good candidates for $x_{k+1}$ by $R_{k+1}$.

Since $\varepsilon<\frac{1}{2h-2}$ and $\gamma < \frac{1}{4h-4}$, there are at least $\frac{\log(h)+c}{\gamma}$ elements in $X \setminus \{x_{1}, x_2, \dots, x_{k}\}$.
The probability that none of them is in $R_{k+1}$ is at most $(1-\gamma)^{\gamma^{-1} (\log(h)+c)}<\frac{e^{-c}}{h}$.
Let $x_{k+1}$ be the smallest element (in the linear extension) such that $x_{k+1} \in R_{k+1} \cap X$ if there is such an element.

The union bound yields the theorem.
\end{proof}

The following proposition shows that Theorem~\ref{thm:optsample} gives the right order of magnitude on the number of samples required for one-sided testing.

We denote by $K_{w_1, w_2, \dots , w_k}$ the complete $h$-partite poset: the set of elements consists of pairwise disjoint antichains $A_i$ of size $w_i$ for $1\leq i \leq k$; and $x\prec y$ for $x\in A_i$ and $y\in A_j$ if and only if $i<j$. Let $K_{h\times w}$ be the shorthand notation for $K_{w, w, \dots , w}$ with $hw$ elements. In particular, $K_{h\times 1}$ is the chain $C_h$.

\begin{proposition}\label{prop:subposet_sharp}
    Given $\varepsilon>0$ and the positive integers $h \geq 2, w \geq 1$ such that $\varepsilon w$ is also an integer, consider the poset $P=K_{\varepsilon w, w, w, \dots ,w}$ with $(\varepsilon + h-1) w$ elements.
    \begin{thmenum}
        \item \label{prop:subposet_sharp_1} Every subposet obtained by the removal of less than $\varepsilon w^2$ edges from $P$ contains $C_h$ as a subposet, hence $P$ is at least $\frac{\varepsilon}{(\varepsilon + h-1)^2}$-far from being $C_h$-free.
        
        \item \label{prop:subposet_sharp_2} For any $0 < c < \varepsilon w$ the probability that a random subset with at most $\frac{c}{2\varepsilon}$ elements does not contain $C_h$ as a subposet is at least $e^{-c}$.
    \end{thmenum}
\end{proposition}

Note that the bound in \ref{prop:subposet_sharp_1} is sharp: if we remove all of the $\varepsilon w^2$ edges between the first two antichains, we obtain a $C_h$-free poset.

\begin{proof}
\ref{prop:subposet_sharp_1} Every edge with an endvertex in the first antichain 
(of size $\varepsilon w$) is contained by exactly 
$w^{h-2}$ chains of height $h$, since we can choose the other elements of the chain from the other antichains arbitrarily. 
On the other hand, an edge not adjacent to the first antichain is contained by $\varepsilon w^{h-2}$ chains. Since $P$ contains $\varepsilon w^h$ chains of height $h$, we need at least $\varepsilon w^2$ edges to cover these.

\ref{prop:subposet_sharp_2} Every subposet isomorphic to $C_h$ has an element in the first antichain. The probability that a subposet on $k \leq \frac{c}{2\varepsilon} < \frac{w}{2}$ elements contains no element of this antichain is

$\prod_{i=1}^{k}\frac{(h-1)w-i+1}{(h-1 +\varepsilon)w-i+1} > \big(\frac{1}{1 +2\varepsilon}\big)^{\frac{c}{2\varepsilon}}>e^{-c}$.
\end{proof}

This gives the right order of magnitude of the number of samples required for the one-sided testing of $C_h$-free posets for every fixed $h$: Theorem~\ref{thm:optsample} shows that using $\left\lceil \frac{4\log(h)+4c+1}{2\varepsilon} \right\rceil$ samples the error probability is at most $e^{-c}$, while Proposition~\ref{prop:subposet_sharp} gives an example where the error is at least $e^{-c}$ when sampling at most $\frac{c}{2\varepsilon}$ elements.

\section{Testing monotone classes of posets}
\label{section:testing}

The following lemma provides a lower bound on the density of the complete $h$-partite poset $K_{h\times w}$ in terms of the density of the chain of length $h$.
The proof is inspired by the counting argument of K\H{o}v\'ari, S\'os and Tur\'an~\cite{KST54} in the proof of the upper bound to the symmetric case of the Zarankiewicz problem (that is, using modern notation, the upper bound on $ex(n,K_{r,r})$).
We use again the notation $[n] = \{1,2,\dots, n\}$.

\begin{lemma} \label{density}
For every poset $P$ and positive integers $h,w$ the inequality
$$t(K_{h \times w}, P) \geq t^{w^2}(C_{h}, P)$$
holds.
\end{lemma}

\begin{proof}
The following two claims imply the lemma.

\begin{claim*}
$$ t(K_{w,1,w,1, \dots }, P) \geq t^w(C_{h}, P)$$
\end{claim*}

\begin{proof}
\renewcommand\qedsymbol{$\blacksquare$}
Note that $K_{w,1,w,1, \dots }$ is the union of $w$ edge-disjoint chains of length $h$ intersecting only on the elements of the even layers (where it has only one element), and a mapping of $K_{w,1,w,1, \dots }$ is a homomorphism if and only if its restriction to every chain is a homomorphism. Consider a mapping of the even layers of $K_{w,1,w,1, \dots }$.
The events that the random mapping gives a homomorphism for chains are conditionally independent for disjoint chains (conditioning on the mapping of the even layers).
Hence, the conditional probability that mapping $w$ elements for every odd layer gives a homomorphism of $K_{w,1,w,1, \dots }$ is the $w^{th}$ power of the probability that mapping only one element for every odd layer gives a homomorphism of the chain $C_h$. We use Jensen's inequality to obtain the required result. Now, we describe this argument more formally.

Let $ \left(x_{i,j}\right)_{i\in [h], j\in [w] \text{ for } i \text{ odd}}$ where $x_{i,j} \in P$ and $ \left(x_{i,1}\right)_{i\in [h] \text{ for } i \text{ even}}$ be chosen uniformly and independently at random in $P$.
\begin{equation*}
\resizebox{\linewidth}{!}{\ensuremath{
  \begin{aligned}
        t& (K_{w,1,w,1, \dots }, P)
        = \mathds{P} _ {\substack{\left(x_{i,1}\right)_{i\in [h]} \text{ for }i \text{ even } \\ \left(x_{i,j}\right)_{i\in [h], j\in [w]} \text{ for } i \text{ odd }}} \left( \Bigg.\forall k \in [h-1], \ell \in [w] \quad \substack{\text{ if } k \text{ odd then } x_{k,\ell} \prec x_{k+1,1} \\ \text{ if } k \text{ even then } x_{k,1} \prec x_{k+1,\ell}}\right)\\
        & = \mathds{E} _ {\substack{\left(x_{i,1}\right)_{i\in [h]} \\ i \text{ even }}} \left[ \mathds{P}_{\substack{\left(x_{i,j}\right)_{i\in [h], j\in [w]}\\ i \text{ odd}}} \left( \Bigg.\forall k \in [h-1], \ell \in [w] \quad \substack{\text{ if } k \text{ odd then } x_{k,\ell} \prec x_{k+1,1} \\ \text{ if } k \text{ even then } x_{k,1} \prec x_{k+1,\ell}} \Big | \left(x_{i,1}\right)_{i\in [h]}, i \text{ even } \right)\right]
    \end{aligned}
    }}
\end{equation*}
Here we split $K_{w,1,w,1, \dots }$ into $w$ edge-disjoint copies of $C_{h}$. Since the events corresponding to elements in the same odd layer are independent, we obtain that this equals

\begin{equation*}
    \begin{split}
    & \mathds{E} _ {\substack{\left(x_{i,1}\right)_{i\in [h]} \\ i \text{ even }}} \left[ \mathds{P} _ {\substack{\left(x_{i,1}\right)_{i\in [h]}\\ i \text{ odd}}} \left( \Bigg.\forall k \in [h-1] \quad x_{k,1} \prec x_{k+1,1} \Big | \left(x_{i,1}\right)_{i\in [h]}, i \text{ even } \right)\right]^w\\
    & \geq \left[\mathds{E} _ {\substack{\left(x_{i,1}\right)_{i\in [h]} \\ i \text{ even }}} \mathds{P} _ {\substack{\left(x_{i,1}\right)_{i\in [h]}\\ i \text{ odd}}} \left( \Bigg.\forall k \in [h-1] \quad x_{k,1} \prec x_{k+1,1} \Big | \left(x_{i,1}\right)_{i\in [h]}, i \text{ even } \right)\right]^w\\
    & = \left[\mathds{P} _ {\left(x_{i,1}\right)_{i\in [h]}} \left( \Bigg.\forall k \in [h-1] \quad x_{k,1} \prec x_{k+1,1} \right)\right]^w = t^w(C_{h}, P),
    \end{split}
\end{equation*}
where we have applied Jensen's inequality.
\end{proof}

\begin{claim*}
$$t(K_{h \times w}, P) \geq t^w(K_{w,1,w,1, \dots }, P)$$
\end{claim*}

\begin{proof}
\renewcommand\qedsymbol{$\blacksquare$}
The proof is very similar to the previous one. Now we use the observation that $K_{h \times w}$ is the union of $w$ edge-disjoint copies of $K_{w,1,w,1, \dots }$ intersecting only on the odd layers (where $K_{w,1,w,1, \dots }$ has $w$ elements), and a mapping of $K_{h \times w}$ is a homomorphism if and only if its restriction to every such copy of $K_{w,1,w,1, \dots }$ is a homomorphism. Consider a mapping of the odd layers of $K_{h \times w}$.
The events that the random mapping gives a homomorphism for copies of $K_{w,1,w,1, \dots }$ are conditionally independent for disjoint copies of $K_{w,1,w,1, \dots }$ (conditioning on the mapping of the odd layers).
Hence, the conditional probability that mapping $w$ elements for every even layer gives a homomorphism of $K_{h \times w}$ is the $w^{th}$ power of the probability that mapping only one element for every even layer gives a homomorphism of $K_{w,1,w,1, \dots }$. We use Jensen's inequality again to obtain the required result.

Let $ \left(x_{i,j}\right)_{i\in [h], j\in [w]}$ be chosen uniformly and independently at random in $P$.
\begin{equation*}
\resizebox{\linewidth}{!}{\ensuremath{
  \begin{aligned}
        t& (K_{h \times w}, P) = \mathds{P}_{\left(x_{i,j}\right)_{i\in [h], j\in [w]}}
        \Big( \forall k \in [h-1], \ell, m \in [w] \quad x_{k,\ell} \prec x_{k+1,m} \Big) \\
        & = \mathds{E}_{\substack{\left(x_{i,j}\right)_{i\in [h], j\in [w]} \\ i \text{ odd }}} \left[ \mathds{P}_{\substack{\left(x_{i,j}\right)_{i\in [h], j\in [w]}\\ i \text{ even}}}
        \left( \forall k \in [h-1], \ell, m \in [w] \quad x_{k,\ell} \prec x_{k+1,m} \Big | \left(x_{i,j}\right)_{i\in [h], j\in [w]}, i \text{ odd } \right)\right].
  \end{aligned}}}
\end{equation*}
Here we split $K_{h\times w}$ into $w$ edge-disjoint copies of $K_{w,1,w,1, \dots}$. Since the events corresponding to elements in the same even layer are independent, we obtain that this equals
\begin{equation*}
    \resizebox{\linewidth}{!}{\ensuremath{
    \begin{aligned}
        & \mathds{E}_{\substack{\left(x_{i,j}\right)_{i\in [h], j\in [w]} \\ i \text{ odd }}} \left[ \mathds{P}_{\substack{\left(x_{i, 1}\right)_{i\in [h]}\\ i \text{ even}}}
        \left( \Bigg.\forall k \in [h-1], \ell \in [w] \quad \substack{\text{ if } k \text{ odd then } x_{k,\ell} \prec x_{k+1,1} \\ \text{ if } k \text{ even then } x_{k,1} \prec x_{k+1,\ell}} \Big | \left(x_{i,j}\right)_{i\in [h], j\in [w]}, i \text{ odd } \right)\right]^w\\
        & \geq \left[ \mathds{E}_{\substack{\left(x_{i,j}\right)_{i\in [h], j\in [w]} \\ i \text{ odd }}} \mathds{P}_{\substack{\left(x_{i,1} \right)_{i\in [h]}\\ i \text{ even}}}
        \left( \Bigg.\forall k \in [h-1], \ell \in [w] \quad \substack{\text{ if } k \text{ odd then } x_{k,\ell} \prec x_{k+1,1} \\ \text{ if } k \text{ even then } x_{k,1} \prec x_{k+1,\ell}} \Big | \left(x_{i,j}\right)_{i\in [h], j\in [w]}, i \text{ odd } \right)\right]^w\\
        & = \left[\mathds{P}_{\substack{\left(x_{i,1}\right)_{i\in [h]} \text{ for }i \text{ even } \\ \left(x_{i,j}\right)_{i\in [h], j\in [w]} \text{ for } i \text{ odd }}}
        \left( \Bigg.\forall k \in [h-1], \ell \in [w] \quad \substack{\text{ if } k \text{ odd then } x_{k,\ell} \prec x_{k+1,1} \\ \text{ if } k \text{ even then } x_{k,1} \prec x_{k+1,\ell}}\right)\right]^w = t^w(K_{w,1,w,1, \dots }, P),
    \end{aligned}}}
\end{equation*}
where we have applied Jensen's inequality.
\end{proof}

The lemma follows.
\end{proof}

\begin{proof}(of Theorem~\ref{thm:posetremoval})
Assume that $t(Q,P)< \left(\frac{\varepsilon}{2}\right)^{hw^2}$.
The poset $Q$ is a subposet of $K_{h\times w}$, so Lemma \ref{density} gives
$t^{w^2}(C_{h},P) \leq t(K_{h\times w},P) \leq t(Q,P)$. These yield $t(C_{h},P)< \left(\frac{\varepsilon}{2}\right)^{h}$,
so by Lemma~\ref{lemma:chainremoval} there is a $C_h$-free subposet $P'$ of $P$ obtained by deleting at most $\varepsilon |P|^2$ edges.
\end{proof}

\begin{proof}
(of Theorem~\ref{thm:indistinguish})
If a poset is $C_h$-free, then it is $\mathcal{P}$-free.

In order to prove the other direction, consider a poset $Q \in \mathcal{P}$ with minimal height $h=h(\mathcal{P})$ and (amongst these) minimal width $w=w(\mathcal{P})$. If a poset $P$ is $Q$-free, then there is no injective homomorphism from $Q$ to $P$. The probability that a $Q\rightarrow P$ mapping is not injective is at most $\binom{|Q|}{2} |P|^{-1}$ since a pair of elements in $Q$ should be mapped to the same element in $P$. Thus, $t(Q,P) \leq |P|^{-1}|Q|^2$. Since $|Q| \leq hw$, Theorem~\ref{thm:posetremoval} shows that one can get a $C_h$-free subposet of $P$ by the removal of 
$2 (h^2 w^2)^{\frac{1}{hw^2}}|P|^{-\frac{1}{hw^2}}|P|^2$ edges.
\end{proof}

\section{Comparability graphs}\label{section:cograph}

We will obtain the same theorems for monotone classes of comparability graphs as for posets: the difference will only be in the hidden constants. These allow the same tests as for posets. For a fixed finite graph $F$, the basic test samples $|V(F)|$ vertices and accepts a graph if these do not span an isomorphic copy of $F$. The following removal lemma shows how many iterations we need to reject comparability graphs $\varepsilon$-far from being $F$-free with probability at least one half while always accepting $F$-free comparability graphs.
Similarly to posets, given two finite graphs $F, G$, the probability that a uniform random mapping from $F$ to $G$ is a homomorphism (i.e., edge-preserving) is denoted by $t(F,G)$.

\begin{theorem}\label{thm:cographremovel}[Polynomial removal lemma for comparability graphs]
Consider an $\varepsilon>0$ and a finite graph $F$ that is not an independent set. For every finite comparability graph $G$, if 
$t(F,G) \leq \left(\frac{\varepsilon}{2}\right)^{\chi(F) \alpha(F)^2}$
then there exists an $F$-free (moreover, $K_{\chi(F)}$-free) spanning subgraph of $G$ that is a comparability graph, obtained by deleting at most $\varepsilon |V(G)|^2$ edges.
\end{theorem}

\begin{proof}
The graph $F$ is a subgraph of the multipartite Tur\'an graph $T$ with $\chi(F)$ classes each of size $\alpha(F)$, hence $t(F,G) \geq t(T,G)$.
There exists a poset $P$ with comparability graph $G$. The height of the poset $P$ is exactly the chromatic number of $G$, and the width of the poset $P$ equals the independence number of $G$. 

Note that $t(T,G) \geq t(K_{\chi(F) \times \alpha(F)},P)$, since we may assume that $T$ is the comparability graph of $K_{\chi(F) \times \alpha(F)}$, hence
every homomorphism of $K_{\chi(F) \times \alpha(F)}$ to $P$ is a comparability-preserving map from $T$ to $G$, i.e., a graph homomorphism.
By Theorem~\ref{thm:posetremoval} 
there exists a $C_{\chi(F)}$-free subposet of $P$ obtained by deleting at most $\varepsilon|P|^2$ edges, and its comparability graph satisfies the conditions of the theorem.
\end{proof}

Note that we did not need to know the underlying poset $P$ to prove the existence of the desired subgraph of $G$.

Given a set of (possibly infinitely many) finite graphs $\mathcal{F}$, we define the chromatic number $\chi(\mathcal{F})$ and the independence number $\alpha(\mathcal{F})$ as follows.
\[\chi(\mathcal{F}) = \min_{F\in \mathcal{F}} \chi(F) \quad \quad \quad \alpha(\mathcal{F}) = \min_{\substack{F\in \mathcal{F}:\\ \chi(F)=\chi(\mathcal{F})}} \alpha(F).\]

\begin{corollary}\label{thm:cograph_easy} [Easy testability for monotone classes of comparability graphs]
Consider a family of finite graphs $\mathcal{F}$ and a graph $F \in \mathcal{F}$ with chromatic number $\chi(\mathcal{F})\geq 2$ and independence number $\alpha(\mathcal{F})$.
For every $\varepsilon>0$ and finite comparability graph $G$, if 
$t(F,G) \leq \left(\frac{\varepsilon}{2}\right)^{\chi(\mathcal{F})\alpha(\mathcal{F})^2}$ then there exists an $\mathcal{F}$-free (moreover, $K_{\chi(\mathcal{F})}$-free) spanning subgraph of G, that is a comparability graph, obtained by deleting at most $\varepsilon |V(G)|^2$ edges.
\end{corollary}

We also give a classification of monotone classes of comparability graphs as we did for posets. Two properties $\Phi_1$ and $\Phi_2$ of graphs are {\it indistinguishable} if for every $\varepsilon>0$ and $i=1,2$ there exists $N$ such that for every graph $G$ on at least $N$ vertices with property $\Phi_i$ there exists a graph $G'$ on the same vertex set with property $\Phi_{3-i}$, obtained by changing at most $\varepsilon |V(G)|^2$ edges of $G$.
Since we are interested in monotone properties, we only need to allow deleting edges.

\begin{theorem}\label{thm:cograph_distinguishable} [Indistinguishability]
Consider a family of finite graphs $\mathcal{F}$. Set $\chi=\chi(\mathcal{F}) \geq 2, \alpha = \alpha(\mathcal{F})$. Comparability graphs with chromatic number at most $(\chi-1)$
are indistinguishable from $\mathcal{F}$-free comparability graphs. Namely, every comparability graph with chromatic number at most $(\chi-1)$ is $\mathcal{F}$-free, and every $\mathcal{F}$-free comparability graph admits a spanning subgraph with chromatic number at most $(\chi-1)$, that is a comparability graph, obtained by the removal of at most
$2 \left(\frac{\chi^2 \alpha^2}{|V(G)|}\right)^{\frac{1}{\chi \alpha^2}}|V(G)|^2$
edges. 
\end{theorem}

\begin{proof}
Clearly, every comparability graph with chromatic number at most $(\chi-1)$ is $\mathcal{F}$-free. On the other hand, given an $\mathcal{F}$-free comparability graph $G$, consider a poset $P$ whose comparability graph is $G$. Theorem~\ref{thm:indistinguish} implies that there is a $C_{\chi}$-free subposet $P'$ obtained by the removal of at most $2 \left(\frac{\chi^2 \alpha^2}{|V(G)|}\right)^{\frac{1}{\chi \alpha^2}}|V(G)|^2$ edges. The comparability graph $G'$ of $P'$ is the desired spanning subgraph of $G$: it is $K_{\chi}$-free, since $P'$ is $C_{\chi}$-free. Hence, $\chi(G') \leq \chi-1$ by the dual of the Dilworth theorem.
\end{proof}

The analog of Algorithm~\ref{test2} is the test sampling a random set of vertices and accepting the graph if the subgraph spanned by them is $K_{\chi}$-free. We need the same number of samples as in the case of posets. The following theorem is a straightforward consequence of Theorem~\ref{thm:optsample}.

\begin{theorem} \label{thm:cograph4}
[On the subgraph test]
Let $\chi\geq 2$ be an integer,
$\varepsilon>0, c>0$ and $G$ a finite comparability graph. If $G$ is $\varepsilon$-far from being $K_{\chi}$-free then a random subset of $\left\lceil \frac{4\log(\chi)+4c+1}{2\varepsilon} \right\rceil$ vertices chosen independently and uniformly at random contains a copy of $K_{\chi}$ with probability at least $1-e^{-c}$.
\end{theorem}

The comparability graph of the poset in Proposition~\ref{prop:subposet_sharp} shows that for any fixed $h$, this bound has the right order of magnitude in $\varepsilon$. As in the case of posets, we can also use the test for $K_{\chi(\mathcal{F})}$-free subgraphs to test a monotone class of comparability graphs $\mathcal{F}$: the probability that we reject an $\mathcal{F}$-free comparability graph is negligible.

\section{Acknowledgements}

Panna Tímea Fekete's Project No.\ 1016492.\ has been implemented with the support provided by the Ministry of Culture and Innovation of Hungary from the National Research, Development and Innovation Fund, financed under the KDP-2020 funding scheme and supported by the ERC Synergy Grant No.\ 810115 -- DYNASNET.
Gábor Kun was supported by the Hungarian Academy of Sciences Momentum Grant no. 2022-58 and the ERC Advanced Grant ERMiD.


\begin{thebibliography}{30}
\bibitem{ABF17}
Alon, Noga and Ben-Eliezer, Omri and Fischer, Eldar. "Testing hereditary properties of ordered graphs and matrices." 2017 IEEE 58th Annual Symposium on Foundations of Computer Science (FOCS). IEEE, 2017.

\bibitem{AFKS00}%[AFKS00]
Alon, Noga and Fischer, Eldar and Krivelevich, Michael and Szegedy, Mario (2000). Efficient testing of large graphs. Combinatorica, 20(4), 451-476.

\bibitem{AS03}%[AS03]
Alon, Noga and Shapira, Asaf (2003). Testing subgraphs in directed graphs. In Proceedings of the thirty-fifth annual ACM symposium on Theory of computing (pp. 700-709).

\bibitem{AS05}%[AS05]
Alon, Noga and Shapira, Asaf (2005). Every monotone graph property is testable. In Proceedings of the thirty-seventh annual ACM symposium on Theory of computing (pp. 128-137).

\bibitem{AS06}
Alon, Noga and Shapira, Asaf (2006). A characterization of easily testable induced subgraphs. Combinatorics, Probability and Computing, 15(6), 791-805.

\bibitem{AF11}%[AF11]
Alon, Noga and Fox, Jacob (2011). Testing perfection is hard. arXiv preprint arXiv:1110.2828.

\bibitem{AF15}%[AF15]
Alon, Noga and Fox, Jacob (2015). Easily testable graph properties. Combinatorics, Probability and Computing, 24(4), 646-657.

\bibitem{AT10}%[AT10]
Austin, Tim and Tao, Terence (2010). Testability and repair of hereditary hypergraph properties. Random Structures \& Algorithms, 36(4), 373-463.

\bibitem{B46}%[B46]
Behrend, Felix A (1946). On sets of integers which contain no three terms in arithmetical progression. Proceedings of the National Academy of Sciences, 32(12), 331-332.

\bibitem{BFHHS}
Bhattacharyya, Arnab and Fischer, Eldar and Hatami, Hamed and Hatami, Pooya and Lovett, Shachar (2013, June). Every locally characterized affine-invariant property is testable. In Proceedings of the forty-fifth annual ACM symposium on Theory of computing (pp. 429-436).

%\bibitem{BY}
%Bhattacharyya, Arnab and Yoshida, Yuichi (2022). Property Testing: Problems and Techniques. Springer Nature.

%\bibitem{F11}%[F11]
%Fox, Jacob (2011). A new proof of the graph removal lemma. Annals of Mathematics, 561-579.

\bibitem{FPZ16}
Fox, Jacob and Pach, János and Suk, Andrew. "A polynomial regularity lemma for semialgebraic hypergraphs and its applications in geometry and property testing." SIAM Journal on Computing 45, no. 6 (2016): 2199-2223.

\bibitem{FW18}%[FW18]
Fox, Jacob and Wei, Fan (2018). Fast property testing and metrics for permutations. Combinatorics, Probability and Computing, 27(4), 539-579.

\bibitem{G20}
Gishboliner, Lior. Efficient Removal Lemmas and Related Problems. Tel Aviv University, 2020.

\bibitem{GS22}
Gishboliner, Lior and Shapira, Asaf. Hypergraph removal with polynomial bounds. arXiv preprint arXiv:2202.07567, 2022.

\bibitem{GT21}
Gishboliner, Lior and Tomon, Istv{\'a}n (2021). Polynomial removal lemmas for ordered graphs. arXiv preprint arXiv:2110.03577.

\bibitem{GGR98}%[GGR98]
Goldreich, Oded and Goldwasser, Shari and Ron, Dana (1998). Property testing and its connection to learning and approximation. Journal of the ACM (JACM), 45(4), 653-750.

\bibitem{G05}%[G05]
Green, Ben (2005). A Szemer\'edi-type regularity lemma in abelian groups, with applications. Geometric \& Functional Analysis GAFA, 15(2), 340-376.

\bibitem{HMPP15}
Hladk{\`y}, Jan and M{\'a}th{\'e}, Andr{\'a}s and Patel, Viresh and Pikhurko, Oleg (2015). Poset limits can be totally ordered. Transactions of the American Mathematical Society, 367(6),4319--4337.

\bibitem{KK14}%[KK14]
Klimo{\v{s}}ov{\'a}, Tereza and Kr{\'a}l', Daniel (2014). Hereditary properties of permutations are strongly testable. In Proceedings of the Twenty-Fifth Annual ACM-SIAM Symposium on Discrete Algorithms (pp. 1164-1173). Society for Industrial and Applied Mathematics.

\bibitem{KST54}
Kővári, P., T Sós, V., and Turán, P. (1954). On a problem of Zarankiewicz. In Colloquium Mathematicum (Vol. 3, pp. 50-57). Polska Akademia Nauk.

\bibitem{KSV12}%[KSV12]
Kr{\'a}l', Daniel and Serra, Oriol and Vena, Llu{\'\i}s (2012). A removal lemma for systems of linear equations over finite fields. Israel Journal of Mathematics, 187, 193-207.

\bibitem{LS08}%[LS08]
Lov{\'a}sz, L{\'a}szl{\'o} and Szegedy, Bal{\'a}zs (2008). Testing properties of graphs and functions. arXiv preprint arXiv:0803.1248.

\bibitem{RS09}%[RS09]
R{\"o}dl, Vojt{\v{e}}ch and Schacht, Mathias (2009). Generalizations of the removal lemma. Combinatorica, 29(4), 467-501.

\bibitem{R53}%[R53]
Roth, Klaus F (1953). On certain sets of integers. J. London Math. Soc, 28(1), 104-109.

\bibitem{RS96}%[RS96]
Rubinfeld, Ronitt and Sudan, Madhu (1996). Robust characterizations of polynomials with applications to program testing. SIAM Journal on Computing, 25(2), 252-271.

\bibitem{RS76}%[RS76]
Ruzsa, Imre and Szemer{\'e}di, Endre (1978). Triple systems with no six points carrying three triangles. Combinatorics (Keszthely, 1976), Coll. Math. Soc. J. Bolyai, 18(939-945), 2.

\end{thebibliography}
\end{document}